\documentclass[oneside,10pt]{article}          
\usepackage[b5paper]{geometry}	    

\usepackage{graphicx}
\usepackage{epstopdf}
\usepackage{amsfonts,amsmath,latexsym,amssymb} 
\usepackage{amsthm}                
\usepackage{mia}
\usepackage{calrsfs}
\usepackage{enumerate}
\usepackage{url}
\usepackage{todonotes}

\usepackage{hyperref}

\newtheorem{theorem}{Theorem}

\newtheorem{proposition}[theorem]{Proposition}

\theoremstyle{definition}

\newtheorem{remark}{Remark}

\renewcommand{\Psi}{\overline{\Phi}}

\newcommand{\R}{\mathbb{R}}

\newcommand{\F}
{\Sigma}

\renewcommand{\L}{\mathcal{L}}

\renewcommand{\le}{\leqslant}
\renewcommand{\ge}{\geqslant}
\begin{document}

\title[A Cauchy--Schwarz-type inequality based on a Lagrange-type identity]{An intertwined Cauchy--Schwarz-type inequality based on a Lagrange-type identity}


\author{Iosif Pinelis}

\address{Department of Mathematical Sciences\\
Michigan Technological University\\
Houghton, Michigan 49931, USA\\
\email{ipinelis@mtu.edu}}

%

\CorrespondingAuthor{Iosif Pinelis}


\date{22 June 2015}                               

\keywords{
Cauchy--Schwarz-type inequality, Lagrange-type identity}

\subjclass{
Primary 26D15; secondary 26D20}



\begin{abstract}
Based on an apparently new Lagrange-type identity, a Cauchy--Schwarz-type inequality is proved. The mentioned identity is obtained by using certain ``macro'' variables; it is hoped that such a method can be used to prove or produce other identities and inequalities. 
\end{abstract}

\maketitle







\section{Result}

Let $a_1,a_2,a_3,b_1,b_2,b_3$ be any real numbers. 
The well-known Lagrange identity (see e.g.\ \cite{greene-krantz}) 
$$(a^2_{1}+a^2_{2}+a^2_{3})(b^2_{1}+b^2_{2}+b^2_{3})=(a_{1}b_{1}+a_{2}b_{2}+a_{3}b_{3})^2+\sum_{1\le i<j\le3}(a_{i}b_{j}-a_{j}b_{i})^2$$
immediately yields the Cauchy--Schwarz inequality (see e.g.\ \cite{steele}) 
$$(a^2_{1}+a^2_{2}+a^2_{3})(b^2_{1}+b^2_{2}+b^3_{3})\ge (a_{1}b_{1}+a_{2}b_{2}+a_{3}b_{3})^2. $$

In this note
we shall prove the following, apparently new Cauchy--Schwarz-type inequality, based on an apparently new Lagrange-type identity. 
\begin{proposition}\label{prop:}
\begin{equation}\label{eq:}
\begin{aligned}
(a^2_{1}+ & b^2_{2}+b^2_{3})(a^2_{2}+b^2_{3}+b^2_{1})(a^2_{3}+b^2_{1}+b^2_{2}) \\ 
\ge & (a_{1}b_{1}+a_{2}b_{2}+a_{3}b_{3})^2\,(b^2_{1}+b^2_{2}+b^2_{3}) \\  +&\tfrac{1}{2}\,\big[b_1^2(a_2b_3-a_3b_2)^2+b_2^2(a_3b_1-a_1b_3)^2+b_3^2(a_1b_2-a_2b_1)^2\big]. 	
\end{aligned}	
\end{equation}
\end{proposition}


Note that -- in distinction with the left-hand side of the Cauchy--Schwarz inequality, with the $a_i$'s and $b_i$'s separated in the two factors there -- the $a_i$'s and $b_i$'s are intertwined in the three factors on the left-hand side of inequality \eqref{eq:}. 

%
%

\section{Proof}
 
\begin{proof}[Proof of Proposition~\ref{prop:}]
Let $\tilde d$ denote the difference between the left- and right-hand sides of inequality \eqref{eq:}, which can then 
be rewritten as $\tilde d\ge0$. 
Note that $\tilde d$ is a polynomial (of degree $6$ in $6$ variables). Therefore, in principle, inequality \eqref{eq:} can be verified completely algorithmically, using one of the suitable known tools. One of these tools is the quantifier elimination by cylindrical algebraic decomposition (see e.g.\ \cite{collins98}), based on the Tarski theory \cite{tarski48}; for instance, in Mathematica this theory is implemented via \verb!Reduce[]! and related commands. Alternatively, one may try some of the various 
Positivstellens\"atze of real algebraic geometry (see e.g.\ \cite{krivine64a,
cassier,handelman85
}), which can provide a so-called certificate of positivity to a polynomial that is indeed positive on a set defined by a system of polynomial inequalities (over $\R$). 
However, our polynomial $\tilde d$ turns out to be too complicated for these tools to succeed without substantial human intervention. 

To prove Proposition~\ref{prop:}, many rounds of rewriting of $\tilde d$ were done -- manually, each round verified with Mathematica. Complete details of this multi-step rewriting can be seen in the 6-page Mathematica notebook 1stRewriting.nb and its pdf \break 
image 1stRewriting.pdf, found in the zip file MathematicaVerfication.zip, which can be downloaded at \url{https://works.bepress.com/iosif-pinelis/22/}. 
After that, to verify inequality \eqref{eq:} in the rewritten form, 
the mentioned Mathematica command \verb!Reduce[]! took about 23 min, which is a very long time for a contemporary computer (with a $3.5$ GHz CPU). One may therefore surmise that a description of the execution of this command would possibly take hundreds or thousands of pages when transcribed into regular mathematical writing. 

Fortunately, a few more rounds of rewriting, presented in the Mathematica notebook 2ndRewriting.nb and its pdf image 2ndRewriting.pdf in the mentioned zip file MathematicaVerfication.zip, yield a key identity, which allows one to prove inequality \eqref{eq:} rather quickly and easily. 


To state this identity, note first that, without loss of generality (wlog), all the $a_i$'s and $b_i$'s are non\-zero. 
%
For $i=1,2,3$, introduce the new, ``macro'' variables  
\begin{equation*}
	x_i:=a_1 a_2 a_3/a_i,\quad y_i:=b_1 b_2 b_3/b_i,\quad p_i:=(x_i-y_i)y_i, \quad z_i:=y_i^2\ge0,  
\end{equation*}
and then 
\begin{equation}\label{eq:0}
c_1:=p_2^2 + p_2 p_3 + p_3^2,\quad c_2:=p_1^2 + p_1 p_3 + p_3^2,\quad c_3:=p_2^2 + p_2 p_1 + p_1^2.  
\end{equation}
Note that $x_1 x_2 x_3=(a_1 a_2 a_3)^2>0$, 
$y_1 y_2 y_3=(b_1 b_2 b_3)^2>0$, $c_1\ge0$, $c_2\ge0$, and $c_3\ge0$. Moreover, 
\begin{equation}\label{eq:1}
(p_1+z_1)(p_2+z_2)(p_3+z_3)\ge0. 
\end{equation} 
The mentioned crucial identity is  
\begin{equation}\label{eq:ident}
	y_1 y_2 y_3\tilde d=d:=
	p_1 p_2 p_3+c_1 z_1+c_2 z_2+c_3 z_3. 	
\end{equation}
As it is clear now, this identity was difficult to obtain. However, it is quite straightforward (but tedious) to verify it. Such a verification is best done using one of a number of available computer algebra programs. E.g., it takes Mathematica only about $0.15$ sec to check identity \eqref{eq:ident}; for details, see the Mathematica notebook checkingTheIdentity.nb and/or its pdf image checkingTheIdentity.pdf in the same zip file, MathematicaVerfication.zip.  


Since $y_1 y_2 y_3>0$, $\tilde d$ equals $d$ in sign. So, it suffices to show that $d\ge0$ -- for any real $p_i$'s, the $c_i$'s as in \eqref{eq:0}, and any nonnegative $z_i$'s satisfying \eqref{eq:1}. 

Note here that without loss of generality $p_1 p_2 p_3<0$ -- otherwise, the desired inequality $d\ge0$ immediately follows because the $c_i$'s and $z_i$'s are nonnegative. So, we may assume that the $p_i$'s are are all nonzero and hence the $c_i$'s are all strictly positive. 

Take any nonzero real $p_i$'s and any nonnegative $z_i$'s such that \eqref{eq:1} holds. Let us then fix those $z_1$ and $z_2$, and let $z_3$ be decreasing as long as $z_3$ remains nonnegative and \eqref{eq:1} holds; clearly, this process can stop only when the value of $z_3$ becomes either $0$ or $-p_3$, and in the latter case we must have $-p_3>0$. 
Moreover, since $c_i>0$ for all $i$, the value of $d$ will not increase after this process is complete. 

We can then proceed similarly by decreasing $z_2$ (instead of $z_3$), and then by decreasing $z_1$. 

Let now $(z_1,z_2,z_3)$ be any minimizer of $d$, subject to the stated conditions on the $z_i$'s. Then it follows from the above reasoning that $z_i\in\{0,-p_i\}$ for each $i=1,2,3$; moreover, if at that $z_i=-p_i$ for some $i$, then we must have $-p_i>0$. 
So, by the symmetry with respect to permutations of the indices, it is enough to consider the following four cases: 

(i) $z_1=-p_1>0$, $z_2=-p_2>0$, $z_3=-p_3>0$; 

(ii) $z_1=-p_1>0$, $z_2=-p_2>0$, $z_3=0$; 

(iii) $z_1=-p_1>0$, $z_2=0$, $z_3=0$; 

(iv) $z_1=0$, $z_2=0$, $z_3=0$. 

In case (i), $\min_{z_1,z_2,z_3}d=-(p_1 + p_2) (p_1 + p_3) (p_2 + p_3)>0$. 

In case (ii),  $\min_{z_1,z_2,z_3}d=-p_1 p_2 (p_1 + p_2) - p_1 p_2 p_3 + (-p_1 - p_2) p_3^2$, which is a convex quadratic polynomial in $p_3$, with discriminant $-p_1 p_2 (4 p_1^2 + 7 p_1 p_2 + 4 p_2^2)<0$, whence again $\min_{z_1,z_2,z_3}d>0$. 

In case (iii), $\min_{z_1,z_2,z_3}d=-p_1 (p_2^2 + p_3^2)>0$. 

In case (iv), condition \eqref{eq:1} becomes $p_1 p_2 p_3\ge0$, which contradicts the assumption $p_1 p_2 p_3<0$. 

Thus, $\min_{z_1,z_2,z_3}d\ge0$ in all feasible cases, and 
\eqref{eq:} is proved. 
\end{proof}

\section{Discussion}

Note that each of the factors on the left-hand side of inequality \eqref{eq:} is the sum of three terms. It would be interesting (but possibly very difficult) to extend this inequality to an ``intertwined'' one similarly involving sums of more than three terms. 

\bigskip

As was noted in the proof of Proposition~\ref{prop:}, wlog all the $b_i$'s are non\-zero. Introducing then $k_i:=a_i/b_i$, we can rewrite inequality \eqref{eq:} as follows: 
\begin{equation}\label{eq:k_i}
\begin{aligned}
(k_1^2 b^2_{1}+ & b^2_{2}+b^2_{3})
(k_2^2 b^2_{2}+b^2_{3}+b^2_{1})
(k_3^2 b^2_{3}+b^2_{1}+b^2_{2}) \\ 
\ge & (k_1 b_{1}^2+k_2 b_{2}^2+k_3 b_{3}^2)^2\,(b^2_{1}+b^2_{2}+b^2_{3}) \\  +&\tfrac{1}{2}\,b^2_{1}b^2_{2}b^2_{3}\,
\big[(k_1-k_2)^2+(k_2-k_3)^2+(k_1-k_3)^2\big] 	
\end{aligned}	
\end{equation}
for all real $b_i$'s and $k_i$'s. 

\begin{remark}\label{rem:}
The constant factor $\frac12$ in \eqref{eq:k_i} and hence in \eqref{eq:} is optimal -- that is, the greatest possible one. Indeed, if the factor $\frac12$ in \eqref{eq:k_i} is replaced by any real  constant $C>\frac12$, then for $k_1=k_2=0$ 
the difference between the left- and right-hand sides of inequality \eqref{eq:k_i} will be 
$(1-2C) b_1^2 b_2^2 b_3^2 k_3^2+\left(b_1^2+b_2^2\right) \left(b_1^2+b_3^2\right)
   \left(b_2^2+b_3^2\right)$, which will go to $-\infty$ as $k_3\to\infty$ if $b_1 b_2 b_3\ne0$. 
\end{remark}

\bigskip

One may also note that \eqref{eq:} immediately implies the following simpler but weaker ``intertwined'' inequality: 
\begin{equation*}
	(a^2_{1}+b^2_{2}+b^2_{3})(a^2_{2}+b^2_{3}+b^2_{1})(a^2_{3}+b^2_{1}+b^2_{2}) 
\ge (a_{1}b_{1}+a_{2}b_{2}+a_{3}b_{3})^2\,(b^2_{1}+b^2_{2}+b^2_{3}). 
\end{equation*}

Inequality \eqref{eq:} was conjectured on the MathOverflow site 
\cite{lagr-MO} and proved there by the author of the present note. 

\section{Conclusion}
Looking back at the cases (i)--(iv) in the proof of Proposition~\ref{prop:} and at Remark~\ref{rem:}, we notice a rather large number of entire varieties of cases of minima and near-minima of $d$ or $\tilde d$. This may at least partially explain the difficulties with using standard methods, such as cylindrical algebraic decomposition and certificates of positivity provided by Positivstellens\"atze, mentioned in the beginning of the proof of Proposition~\ref{prop:}. 

The ``macro''-variables method, demonstrated in this note, may turn out to be useful in other settings where the other methods are not feasible. It would be of great interest if computers could be taught this method, as they have been taught the mentioned standard methods. 


\bibliographystyle{abbrv}


\bibliography{C:/Users/ipinelis/Documents/pCloudSync/mtu_pCloud_02-02-17/bib_files/citations04-02-21}

\def\cprime{$'$} \def\polhk#1{\setbox0=\hbox{#1}{\ooalign{\hidewidth
  \lower1.5ex\hbox{`}\hidewidth\crcr\unhbox0}}}
  \def\polhk#1{\setbox0=\hbox{#1}{\ooalign{\hidewidth
  \lower1.5ex\hbox{`}\hidewidth\crcr\unhbox0}}}
  \def\polhk#1{\setbox0=\hbox{#1}{\ooalign{\hidewidth
  \lower1.5ex\hbox{`}\hidewidth\crcr\unhbox0}}} \def\cprime{$'$}
  \def\polhk#1{\setbox0=\hbox{#1}{\ooalign{\hidewidth
  \lower1.5ex\hbox{`}\hidewidth\crcr\unhbox0}}} \def\cprime{$'$}
  \def\polhk#1{\setbox0=\hbox{#1}{\ooalign{\hidewidth
  \lower1.5ex\hbox{`}\hidewidth\crcr\unhbox0}}} \def\cprime{$'$}
  \def\cprime{$'$}
\begin{thebibliography}{1}

\bibitem{cassier}
G.~Cassier.
\newblock Probl\`eme des moments sur un compact de {${\bf R}^{n}$} et
  d\'ecomposition de polyn\^omes \`a plusieurs variables.
\newblock {\em J. Funct. Anal.}, 58(3):254--266, 1984.

\bibitem{collins98}
G.~E. Collins.
\newblock Quantifier elimination for real closed fields by cylindrical
  algebraic decomposition.
\newblock In {\em Quantifier elimination and cylindrical algebraic
  decomposition ({L}inz, 1993)}, Texts Monogr. Symbol. Comput., pages 85--121.
  Springer, Vienna, 1998.

\bibitem{greene-krantz}
R.~E. Greene and S.~G. Krantz.
\newblock {\em Function theory of one complex variable}, volume~40 of {\em
  Graduate Studies in Mathematics}.
\newblock American Mathematical Society, Providence, RI, third edition, 2006.

\bibitem{handelman85}
D.~Handelman.
\newblock Positive polynomials and product type actions of compact groups.
\newblock {\em Mem. Amer. Math. Soc.}, 54(320):xi+79, 1985.

\bibitem{krivine64a}
J.-L. Krivine.
\newblock Anneaux pr\'eordonn\'es.
\newblock {\em J. Analyse Math.}, 12:307--326, 1964.

\bibitem{lagr-MO}
MathOverflow.
\newblock An inequality concerning {L}agrange's identity.
\newblock URL:http://mathoverflow.net/q/239243 (version: 2016-06-22).

\bibitem{steele}
J.~M. Steele.
\newblock {\em The {C}auchy-{S}chwarz master class}.
\newblock MAA Problem Books Series. Mathematical Association of America,
  Washington, DC; Cambridge University Press, Cambridge, 2004.
\newblock An introduction to the art of mathematical inequalities.

\bibitem{tarski48}
A.~Tarski.
\newblock {\em A {D}ecision {M}ethod for {E}lementary {A}lgebra and
  {G}eometry}.
\newblock RAND Corporation, Santa Monica, Calif., 1948.

\end{thebibliography}


\end{document}